\documentclass[11pt]{article}

\date{{\color{black}April 16, 2026 (final version)}}

\PassOptionsToPackage{dvipsnames,svgnames,x11names}{xcolor}

\usepackage{mathpazo}   

\usepackage{nicematrix}
\usepackage{array}
\usepackage{booktabs}
\usepackage{siunitx}

\usepackage{tikz}
\definecolor{myseagreen}{HTML}{3FBC9D}
\usepackage[color=myseagreen]{todonotes}

\usepackage{soul}
\usepackage{nameref}
\usepackage{mathtools}
\usepackage{empheq}
\usepackage{comment}
\usepackage[shortlabels,inline]{enumitem}
\setlist[enumerate]{nosep}
\usepackage[colorlinks=true,
linkcolor=refkey,
urlcolor=lblue,
citecolor=magenta]{hyperref}
\usepackage[doc,wmm,hhb]{optional}
\usepackage{xcolor}

\usepackage{float}
\usepackage{soul}
\usepackage{graphicx}
\definecolor{labelkey}{rgb}{0,0.08,0.45}
\definecolor{refkey}{rgb}{0,0.6,0.0}
\definecolor{Brown}{rgb}{0.45,0.0,0.05}
\definecolor{lime}{rgb}{0.00,0.8,0.0}
\definecolor{lblue}{rgb}{0.5,0.5,0.99}
\definecolor{OliveGreen}{rgb}{0,0.6,0}
\definecolor{tyrianpurple}{rgb}{0.4, 0.01, 0.24}

\colorlet{hlcyan}{cyan!30}

\usepackage{stmaryrd} 
\usepackage{amssymb}

\makeatletter
\def\namedlabel#1#2{\begingroup
   \def\@currentlabel{#2}%
   \label{#1}\endgroup
}
\makeatother

\newcommand{\seppthree}{\setlength{\itemsep}{-3pt}}

\usepackage[margin=1in,footskip=0.25in]{geometry}
\providecommand{\siff}{\Leftrightarrow}
\newcommand{\weakly}{\ensuremath{\:{\rightharpoonup}\:}}

\newcommand{\nnn}{\ensuremath{{n\in{\mathbb N}}}}

\newcommand{\thalb}{\ensuremath{\tfrac{1}{2}}}
\newcommand{\menge}[2]{\big\{{#1}~\big |~{#2}\big\}}

\newcommand{\fenv}[1]%
{\ensuremath{\,\overrightarrow{\operatorname{env}}_{#1}}}
\newcommand{\benv}[1]%
{\ensuremath{\,\overleftarrow{\operatorname{env}}_{#1}}}

\newcommand{\scal}[2]{\left\langle{#1},{#2}  \right\rangle}

\newcommand{\exi}{\ensuremath{\exists\,}}

\newcommand{\RR}{\ensuremath{\mathbb R}}

\newcommand{\NN}{\ensuremath{\mathbb N}}

\newcommand{\inte}{\ensuremath{\operatorname{int}}}

\newcommand{\aff}{\ensuremath{\operatorname{aff}\,}}

\newcommand{\cconv}{\ensuremath{\overline{\operatorname{conv}}\,}}
\newcommand{\caff}{\ensuremath{\overline{\operatorname{aff}}\,}}

\newcommand{\Fix}{\ensuremath{\operatorname{Fix}}}

\newcommand{\pinf}{\ensuremath{+\infty}}

\providecommand{\fejer}{Fej\'{e}r}
\providecommand{\wrt}{with respect to}
{\begin{list}{}{%
\settowidth{\labelwidth}{\textrm{#1~}}%
\setlength{\leftmargin}{\labelwidth+\labelsep}}}
{\end{list}}
\usepackage{amsthm}
\usepackage{aliascnt}
\makeatletter
\def\th@plain{%
	\thm@notefont{}
	\itshape 
}
\def\th@definition{%
	\thm@notefont{}
	\normalfont 
}
\makeatother
\usepackage[capitalize,nameinlink]{cleveref}
\crefname{equation}{}{equations}
\crefname{chapter}{Appendix}{chapters}
\crefname{item}{}{items}
\crefname{enumi}{}{}

\newtheorem{theorem}{Theorem}[section]
\newaliascnt{lemma}{theorem}
\newtheorem{lemma}[lemma]{Lemma}
\aliascntresetthe{lemma}

\newaliascnt{corollary}{theorem}
\newtheorem{corollary}[corollary]{Corollary}
\aliascntresetthe{corollary}

\newaliascnt{proposition}{theorem}
\newtheorem{proposition}[proposition]{Proposition}
\aliascntresetthe{proposition}

\newaliascnt{definition}{theorem}
\newtheorem{definition}[definition]{Definition}
\aliascntresetthe{definition}

\newaliascnt{example}{theorem}
\newtheorem{example}[example]{Example}
\aliascntresetthe{example}
\newtheorem{ex}[example]{Example}

\newaliascnt{fact}{theorem}
\newtheorem{fact}[fact]{Fact}
\aliascntresetthe{fact}

\newaliascnt{remark}{theorem}
\newtheorem{remark}[remark]{Remark}
\aliascntresetthe{remark}

\crefname{theorem}{Theorem}{Theorems}
\Crefname{theorem}{Theorem}{Theorems}
\crefname{lemma}{Lemma}{Lemmas}
\Crefname{lemma}{Lemma}{Lemmas}
\crefname{corollary}{Corollary}{Corollaries}
\Crefname{corollary}{Corollary}{Corollaries}
\crefname{proposition}{Proposition}{Propositions}
\Crefname{proposition}{Proposition}{Propositions}
\crefname{definition}{Definition}{Definitions}
\Crefname{definition}{Definition}{Definitions}
\crefname{example}{Example}{Examples}
\Crefname{example}{Example}{Examples}
\crefname{fact}{Fact}{Facts}
\Crefname{fact}{Fact}{Facts}
\crefname{remark}{Remark}{Remarks}
\Crefname{remark}{Remark}{Remarks}

\crefalias{lem}{lemma}
\crefalias{cor}{corollary}
\crefalias{prop}{proposition}
\crefalias{defn}{definition}
\crefalias{ex}{example}
\crefalias{rem}{remark}
\crefalias{thm}{theorem}

\providecommand{\norm}[1]{\lVert#1\rVert}

\providecommand{\LA}{\Leftarrow}
\providecommand{\RA}{\Rightarrow}

\providecommand{\RR}{\mathbb{R}}

\providecommand{\aff}{\operatorname{aff}}

\providecommand{\NN}{\mathbb{N}}

\providecommand{\RR}{\mathbb{R}}
\providecommand{\NN}{\mathbb{N}}

\definecolor{myblue}{rgb}{0.9,0.9,0.98}

\allowdisplaybreaks 

\usepackage{relsize}

\newcommand{\crefpart}[2]{%
  \hyperref[#2]{\namecref{#1}~\labelcref*{#1}~\ref*{#2}}%
}

\author{
Aleksandr Arakcheev\thanks{
Mathematics, University
of British Columbia,
Kelowna, B.C.\ V1V~1V7, Canada. E-mail:
 \texttt{aleksandr.arakcheev@ubc.ca}.}
~~~~and~~
Heinz H.\ Bauschke\thanks{
Mathematics, University
of British Columbia,
Kelowna, B.C.\ V1V~1V7, Canada. E-mail:
\texttt{heinz.bauschke@ubc.ca}.}
}

\title{\textsf{
  On Opial's Lemma 
}
}

\begin{document}

\maketitle

{\color{black}
\begin{center}
\emph{Dedicated to the memory of H\'edy Attouch, whose extraordinary mathematical vision,\\ profound influence, and generous spirit continue to inspire our community}
\end{center}
}

\begin{abstract}
Opial's Lemma is a fundamental result in the convergence analysis of sequences generated by optimization 
algorithms in real Hilbert spaces. 
We introduce the concept of Opial sequences--—sequences for which the limit of the distance to each point in a given set exists.    
We systematically derive properties of Opial sequences, contrasting them with the well-studied 
Fej\'er monotone sequences, and establish conditions for weak and strong convergence. 
Key results include characterizations of weak convergence via weak cluster points (reaffirming Opial's Lemma), strong convergence via strong cluster points, and the behavior of projections onto Opial sets in terms of asymptotic centers. 
Special cases and examples are provided to highlight 
the subtle differences in convergence behaviour and projection properties compared to the Fej\'er monotone case.
\end{abstract}

{ 
\small
\noindent
{\bfseries 2020 Mathematics Subject Classification:}
{Primary 
47H09, 
47J26, 
90C25; 
Secondary 
47H05, 
65K10. 
%
}

}

\section{Introduction}

Throughout this paper, 
\begin{equation}
\text{$X$ is a real Hilbert space, 
with inner product $\scal{\cdot}{\cdot}$ and induced norm 
$\|\cdot\|$.
}
\end{equation}
Many optimization problems boil down to finding a point in some 
solution set $C$, a subset of $X$.
Often, a solution is found iteratively via some sequence 
$(x_n)_\nnn$ generated by an optimization algorithm. 
The following result is key in the convergence analysis 
and part of the folklore:

\begin{fact}[Opial's Lemma]
\label{f:OL}
Let $C$ be a nonempty subset of $X$ 
and let $(x_n)_\nnn$ be a sequence in $X$ 
such that 
\begin{equation}
\label{e:Os}
(\forall c\in C)\quad \text{$\lim_n\|x_n-c\|$ exists.}
\end{equation}
If all weak cluster points of $(x_n)_\nnn$ lie in $C$, 
then $(x_n)_\nnn$ converges weakly to some point in $C$.
\end{fact}

We provide a proof of \cref{f:OL} in 
\cref{c:Opialweakchar} below. 
The ingredients for the proof of \cref{f:OL} 
can be traced back to a classical paper by
Opial \cite{Opial}; however, the statement is not explicitly
recorded there. 
For an explicit proof of \cref{f:OL}, see 
Peypouquet's \cite[Lemma~5.2]{Peypouquet}; 
however, he in turn gives credit to 
Baillon's thesis (see \cite[Chapitre~6]{BaillonThesis}). Oftentimes, 
the sequence $(x_n)_\nnn$ satisfies the even stronger property
that $(x_n)_\nnn$ is \emph{Fej\'er monotone with respect to} $C$, i.e., 
\begin{equation}
(\forall c\in C)(\forall\nnn)\quad
\|x_{n+1}-c\|\leq\|x_n-c\|.
\end{equation}
In the \fejer\ monotonicity case, \cref{f:OL} was proved 
by Browder in 1967 (see \cite[Lemma~6]{Browder67}) and 
Browder states that it is essentially due 
Opial \cite{Opial} and implicit in earlier work 
by Sch\"afer \cite{Schafer}.
In many applications, $C$ is typically the (nonempty) 
fixed point set of 
a nonexpansive mapping $T\colon X\to X$, i.e., 
$(\forall x\in X)(\forall y\in X)$ $\|Tx-Ty\|\leq \|x-y\|$, 
and $(x_n)_\nnn = (T^nx_0)_\nnn$. 
In this case, $(T^nx_0)_\nnn$ is indeed \fejer\ monotone 
\wrt\ $\Fix T$. 
Many optimization algorithms fit this \fejer\ monotonicity framework; see, e.g., \cite{BC2017} which also contains 
basic properties of \fejer\ monotone sequences 
in \cite[Section~5.1]{BC2017}. 
Nonetheless, the term Opial's Lemma also refers precisely 
to \cref{f:OL} --- for recent instances, see, e.g., 
work by Attouch and Cabot \cite[Lemma~2.5]{AttCab}
and by Bo\c{t}, Csetnek, and Nguyen \cite[Lemma~A.3]{BCN25}. 

Somewhat surprisingly, sequences satisfying 
\cref{e:Os} have not been studied systematically. 
We thus introduce the following:

\begin{definition}[Opial sequence and Opial set]
\label{d:Os}
Let $C$ be a nonempty subset of $X$. 
We say that a sequence $(x_n)_\nnn$ in $X$ 
is \emph{Opial \wrt\ $C$} if 
\begin{equation}
(\forall c\in C)\quad \text{$\lim_n\|x_n-c\|$ exists.}
\end{equation}
We occasionally refer to $C$ as the 
corresponding \emph{Opial set}. 
\end{definition}

Clearly, if $(x_n)_\nnn$ is a sequence that is \fejer\ monotone \wrt\ $C$, 
then it is also Opial \wrt\ $C$. 

We point out that various notions have been proposed 
that lie between \fejer\ monotone and Opial sequences --- 
indeed, 
for a sequence $(x_n)_\nnn$ in $X$ and a nonempty subset 
$C$ of $X$, we have:
\begin{itemize}
\seppthree
\item $(x_n)_\nnn$ is \emph{\fejer* monotone \wrt\ $C$}
if 
\begin{equation}
(\forall c\in C)(\exi N\in\NN)(\forall n\geq N)
\quad \|x_{n+1}-c\|\leq\|x_n-c\|.
\end{equation}
This notion was introduced by Behling, Bello-Cruz, Iusem, Liu, and Santos 
in \cite{BBCILS} to analyze circumcenter methods, 
and further analyzed by Behling, Bello-Cruz, Iusem, 
Ribeiro, and Santos \cite{BBCIRS}. 
{ See also the preprint\footnote{{ The preprint \cite{Fstarlong} was posted on arXiv about 9 months after the first version of this manuscript was posted and submitted. While this manuscript 
focuses on Opial sequences, the complementary preprint \cite{Fstarlong} deals with 
the relationships between \fejer, \fejer*, and quasi-\fejer\ sequences.}} \cite{Fstarlong} for an 
in-depth study of this notion.}
\item $(x_n)_\nnn$ is \emph{quasi-\fejer\ monotone \wrt\ $C$ of}
\begin{subequations}
\begin{itemize}
\item \emph{type~I} if 
$
(\exi (\varepsilon_n)_\nnn\in\ell^1_+)(\forall c\in C)
(\forall\nnn)
\quad \|x_{n+1}-c\|\leq\|x_n-c\|+\varepsilon_n$;
\item \emph{type~II} if 
$(\exi (\varepsilon_n)_\nnn\in\ell^1_+)(\forall c\in C)
(\forall\nnn)
\quad \|x_{n+1}-c\|^2\leq\|x_n-c\|^2+\varepsilon_n$;
\item \emph{type~III} if 
$(\forall c\in C)(\exi (\varepsilon_n)_\nnn\in\ell^1_+)
(\forall\nnn)
\quad \|x_{n+1}-c\|^2\leq\|x_n-c\|^2+\varepsilon_n$.
\end{itemize}
\end{subequations}
The notions were introduced and systematically studied by 
Combettes in \cite{Comb01qF} to which we refer the reader for 
historical comments and applications. 
{ (See also the preprint \cite{Fstarlong}.)} 
{ Indeed, many optimization algorithms generate \fejer\ or 
at least quasi-\fejer\ sequences. Perhaps the most prominent example is 
the Krasnoselskii-Mann iteration, which generates a sequence $(x_n)_\nnn$ via 
\begin{equation}
x_{n+1} = (1-\alpha_n)x_n + \alpha_n Tx_n,
\end{equation}
for a nonexpansive operator $T\colon X\to X$ and some suitably chosen sequence of relaxation parameters.
One typically shows that $(x_n)_\nnn$ is Opial \wrt\ $\Fix T$ and that 
all weak cluster points of $(x_n)_\nnn$ lie in $\Fix T$, which yields 
weak convergence of $(x_n)_\nnn$ thanks to \cref{f:OL}. 
}
\end{itemize}

It is clear that \fejer* monotone sequences are Opial. 
The fact that every quasi-\fejer\ monotone sequence is Opial 
is known (see \cite[Section~3]{Comb01qF}) and a consequence 
of the following:

\begin{fact}[Robbins-Siegmund]$\negthinspace\negthinspace\negthinspace$\footnote{\cref{f:RS} is 
{\color{black} a non-stochastic version}
 of a classical probabilistic result by Robbins and Siegmund \cite{RoSi}. 
See also \cite[Lemma~11 on page~50]{Polyak}, 
\cite[Lemma~3.1]{Comb01qF}, \cite[Lemma~5.31]{BC2017}, as well as
\cite[Theorem~3.2]{NePo} for a very recent quantitative version.}
\label{f:RS}
Let $(\alpha_n)_\nnn,(\beta_n)_\nnn,(\delta_n)_\nnn,
(\varepsilon_n)_\nnn$ be sequences
of nonnegative real numbers such that 
$\sum_\nnn\delta_n<\pinf$ and 
$\sum_\nnn\varepsilon_n<\pinf$ and 
\begin{equation}
(\forall\nnn)\quad \alpha_{n+1} \leq 
(1+\delta_n)\alpha_n -\beta_n+\varepsilon_n
\end{equation}
Then $(\alpha_n)_\nnn$ is convergent and 
$\sum_\nnn \beta_n<\pinf$.
\end{fact}

Let us point out that there exist sequences that are Opial
but are neither \fejer* monotone nor quasi-\fejer\ monotone:

\begin{example}
Suppose that $X=\RR$ and set 
\begin{equation}
(x_n)_\nnn := \left(0,\tfrac{1}{\sqrt{1}},0,\tfrac{1}{\sqrt{2}},0,\tfrac{1}{\sqrt{3}},\ldots\right). \end{equation}
Then $x_n\to 0$, $(x_n)_\nnn$ is Opial \wrt\ $\{0\}$, 
but $(x_n)_\nnn$ is neither \fejer* monotone 
nor quasi-\fejer\ monotone \wrt\ $\{0\}$  
{\color{black} 
(because $\sum_{n\geq 1}1/\sqrt{n}$ diverges to $+\infty$).}

\end{example}

Having seen various notions between \fejer\ monotone and 
Opial sequences, we now state that 
\emph{the goal of this paper is to derive useful properties of
Opial sequences and contrast them to properties 
of \fejer\ monotone sequences}.

We now summarize some key results in the following table. 
Here $(x_n)_\nnn$ is a sequence in $X$ and 
$C$ is a nonempty subset of $X$.
We denote the sets of weak and strong cluster points 
of $(x_n)_\nnn$ by 
$\mathcal{W}$ and 
$\mathcal{S}$, respectively. 
{ When $C$ is closed and convex, then} 
we denote the strong cluster points of $(P_Cx_n)_\nnn$ 
by $\mathcal{S}((P_Cx_n)_\nnn)$; in addition, 
the (well-defined) asymptotic center (see \cref{d:ac} below) of 
$(x_n)_\nnn$ with respect to $C$ is denoted by 
$\widehat{c}$.
Let us explain how one reads the table:
The first column lists a possible additional assumption.
The second and third column present the strongest 
conclusions 
under this extra assumption. 
We write ``same'' in the third column if the Opial 
conclusion is identical to the \fejer\ conclusion, 
and ``when'' when the additional assumption is required to
reach the same conclusion. 
The fourth column proves pointers to the 
relevant conclusion in the Opial case. 

\begin{center}
\small
\begin{NiceTabular}[c]{@{}>{\raggedright\arraybackslash}p{0.14\textwidth}>{\raggedright\arraybackslash}p{0.28\textwidth}>{\raggedright\arraybackslash}p{0.28\textwidth}>{\raggedright\arraybackslash}p{0.16\textwidth}@{}}
\CodeBefore
\rowcolor{gray!15}{1-1} 
\rowcolors{3}{gray!10}{}
\Body
\toprule
\Block{1-1}{Assumption} &
\Block{1-1}{$(x_n)_\nnn$ is \fejer\ wrt $C$} & 
\Block{1-1}{$(x_n)_\nnn$ is Opial wrt $C$} & 
\Block{1-1}{Notes} \\
\midrule
none &  $(\mathcal{W}-\mathcal{W})\subseteq (C-C)^\perp$ 
& same 
& \cref{t:old} \\
$\mathcal{W}\subseteq C$
&  $x_n\weakly$ some point in $C$ 
&  same 
& \cref{c:Opialweakchar}\\
$\mathcal{S}\cap C\neq\varnothing$
&  $x_n\to$ some point in $C$ 
&  same 
& \cref{p:Opialstrong}\\
$\inte C\neq\varnothing$
&  $x_n\to $ some point in $X$ 
&  $x_n\weakly$ some point in $X$ 
& \cref{c:largeC}\\
none
&  $(x_n)_\nnn$ is \fejer\ wrt $\cconv C$
&  $(x_n)_\nnn$ is Opial wrt $\caff C$
& \cref{p:enlarge}\\
$C=X$
&  $(x_n)_\nnn = (x_0)_\nnn$ 
&  $x_n\weakly$ some point in $X$ 
& \cref{c:X-Opial=>weak}\\
{$C=\cconv C$}
&  $P_Cx_n\to \widehat{c}$ 
&  {$P_Cx_n\weakly \widehat{c}$} when $d_C(x_n)\to 0$
& \cref{t:dto0}\\
{$C=\cconv C$}
&  $P_Cx_n\to \widehat{c}$ 
&  $\mathcal{S}((P_Cx_n)_\nnn) \subseteq  \{\widehat{c}\}$ 
& \cref{p:badlung}\\
{$C=\cconv C$}
&  $P_Cx_n\weakly  \widehat{c}$ 
&  when $C$ is affine
& \cref{t:nice}\\
{$C=\cconv C$}
&  $(d_C(x_n))_\nnn$ is convergent 
&  when $\overline{\{x_n\}_\nnn}$ is compact
& \cref{p:saturday}\\
\bottomrule
\end{NiceTabular}
\normalsize
\end{center}

The sharpness of our results is illustrated by 
various examples provided in this paper. 

The remainder of the paper is organized as follows. 
In \cref{s:two}, we provide basic properties of 
Opial sequences. 
After reviewing results on the asymptotic center, 
we study the projections of Opial sequences in 
\cref{s:three}. 
Finally, the notation we employ is fairly standard and 
largely follows \cite{BC2017} to which we also refer 
the reader for background material.

\section{Weak and strong convergence of Opial sequences}

\label{s:two}

We now systematically derive properties of 
Opial sequences (which we defined in \cref{d:Os}).

\begin{proposition}[basic properties]
\label{p:opial}
Let $(x_n)_\nnn$ be a sequence in $X$ that is Opial \wrt\ a nonempty subset $C$ of $X$. 
Then the following hold:
\begin{enumerate}
\item 
\label{p:opial1}
$(x_n)_\nnn$ is bounded.
\item 
\label{p:opial2}
$\varlimsup_n d_C(x_n)\leq \inf_{c\in C}\lim_n\norm{x_n-c}$.  
\end{enumerate}
\end{proposition}
\begin{proof}
\cref{p:opial1}:
By definition, for every $c\in C$, 
the $\lim_{n}\norm{x_n-c}$ exists.    
Hence $(\|x_n-c\|)_\nnn$ is bounded and thus 
$(x_n)_\nnn$ is bounded. 

\cref{p:opial2}:
Let $c\in C$. 
Then $(\forall\nnn)$ $d_C(x_n)\leq \|x_n-c\|$.
Hence $\varlimsup_n d_C(x_n) \leq \varlimsup_n \|x_n-c\| 
= \lim_{n}\|x_n-c\|$. 
Now take the infimum over $c\in C$. 
\end{proof}

Throughout this paper, we will repeatedly 
use the fact that every 
Opial sequence is bounded. 
Because of this, we are interested to learn about 
the location of weak cluster points: 

\begin{theorem}[location of weak cluster points]
\label{t:old}
Let $(x_n)_\nnn$ be a sequence in $X$ that is 
Opial \wrt\ a nonempty subset $C$ of $X$. 
Let $w_1$ and $w_2$ be two weak cluster points of 
$(x_n)_\nnn$. 
Then
\begin{equation}
w_1-w_2\in (C-C)^\perp. 
\end{equation}
\end{theorem}
\begin{proof}
(We follow the proof for \fejer\ monotone
 sequences presented in \cite[Theorem~6.2.2(ii)]{HBThesis}.)
By \cref{p:opial}\cref{p:opial1}, $(x_n)_\nnn$ is bounded.
Hence there exists subsequences 
$(x_{k_n})_\nnn$ and 
$(x_{l_n})_\nnn$ of $(x_n)_\nnn$ such that 
\begin{equation}
\label{e:0221a}
x_{k_n}\weakly w_1\;\text{and}\; x_{l_n}\weakly w_2.
\end{equation}
Let $c_1$ and $c_2$ be in $C$. 
Because $(x_n)_\nnn$ is Opial with respect to $C$, 
we have 
\begin{equation*}
  \lim_{n}\norm{x_n-c_1}\;\text{and}\;
  \lim_{n}\norm{x_n-c_2}\;\text{exist.}
\end{equation*}
Hence $(\|x_n-c_1\|^2)_\nnn$ and $(\|x_n-c_2\|^2)_\nnn$ are convergent 
sequences, and so is their  difference
\begin{equation*}
  (\|x_n-c_1\|^2-\|x_n-c_2\|^2)_\nnn
\end{equation*}
After expanding and simplifying, we deduce that 
$(-2\scal{x_n}{c_1-c_2})_\nnn$ is convergent, and so 
\begin{equation}
\label{e:0221b}
\lim_{n} \scal{x_n}{c_1-c_2}\;\text{exists.}
\end{equation}
Thus, taking 
the limit in \cref{e:0221b} along the subsequences from 
\cref{e:0221a}, we obtain 
\begin{equation*}
\scal{w_1}{c_1-c_2}=\scal{w_2}{c_1-c_2}. 
\end{equation*}
Therefore, 
$\scal{w_1-w_2}{c_1-c_2}=0$, as claimed.
\end{proof}

\begin{corollary}[Opial's Lemma]
\label{c:Opialweakchar}
Let $(x_n)_\nnn$ be a sequence in $X$ that is 
Opial \wrt\ a nonempty subset $C$ of $X$. 
Then: $(x_n)_\nnn$ converges weakly to some point in $C$
$\siff$
all weak cluster points of $(x_n)_\nnn$ lie in $C$.
\end{corollary}
\begin{proof}
``$\RA$'': Clear.
``$\LA$'': By \cref{p:opial}\cref{p:opial1}, 
the sequence $(x_n)_\nnn$ is bounded.
Let $w_1$ and $w_2$ be two weak cluster points of $(x_n)_\nnn$.
By assumption, $\{w_1,w_2\}\subseteq C$ and so 
$w_1-w_2\in C-C$.
On the other hand, by \cref{t:old},
$w_1-w_2\in (C-C)^\perp$.
Altogether $w_1-w_2\in (C-C)\cap (C-C)^\perp=\{0\}$ and so 
$w_1=w_2$.
\end{proof}

\cref{c:Opialweakchar} provides a convenient characterization 
of weak convergence to an element in the Opial set.
We now present the counterpart for strong convergence:

\begin{proposition}
\label{p:Opialstrong}
Let $(x_n)_\nnn$ be a sequence in $X$ that is 
Opial \wrt\ a nonempty subset $C$ of $X$. 
Then: $(x_n)_\nnn$ converges strongly to some point in $C$
$\siff$ there exists a cluster point of $(x_n)_\nnn$ in $C$.
\end{proposition}
\begin{proof}
``$\Rightarrow$'': Clear. 
``$\Leftarrow$'': 
Suppose $c$ is a cluster point of $(x_n)_\nnn$ that lies in $C$. 
Then there exists a subsequence $(x_{k_n})_\nnn$ of $(x_n)_\nnn$ 
such that 
\begin{equation*}
x_{k_n}\to c\in C. 
\end{equation*}
Because $(x_n)_\nnn$ is Opial \wrt\ $C$, the 
$\ell := \lim_{n\to\infty}\norm{x_{n}-c}$ exists.
Of course, $\norm{x_{k_n}-c}\to\ell$ as well. 
But $x_{k_n}-c\to 0$, so $\ell=0$. 
It follows that $\norm{x_n-c}\to 0$ and 
therefore $x_n\to c$. 
\end{proof}

\begin{corollary}
Let $(x_n)_\nnn$ be a sequence in $X$ that is 
Opial \wrt\ $X$ and such that 
$\{x_n\}_\nnn$ is relatively compact,
 i.e., $\overline{\{x_n\}_\nnn}$ is compact. 
Then $(x_n)_\nnn$ is strongly convergent. 
\end{corollary}
\begin{proof}
The relative compactness assumption guarantees
that $(x_n)_\nnn$ has a strong cluster point. 
Now apply \cref{p:Opialstrong} with $C=X$. 
\end{proof}

Loosely speaking, the next result says that if 
the Opial set $C$ is ``sufficiently large'', then the Opial sequence converges weakly: 

\begin{corollary}
\label{c:largeC}
Let $(x_n)_\nnn$ be a sequence in $X$ that is Opial \wrt\ a nonempty subset $C$ of $X$. 
If $\caff C =X$, then $(x_n)_\nnn$ converges weakly 
to a point in $X$. 
\end{corollary}
\begin{proof}
By assumption, 
the closed linear span of $C-C$ is equal to entire space $X$. 
But then $(C-C)^\perp = X^\perp = \{0\}$ and so 
all weak cluster points of $(x_n)_\nnn$ must coincide 
by \cref{t:old}. 
\end{proof}

\begin{remark}
Assume that a sequence $(x_n)_\nnn$ in $X$ 
is Opial \wrt\ a nonempty subset $C$
of $X$ and that 
one of the following holds:
\begin{enumerate}
\item $\inte C\neq\varnothing$, or 
\item $X=\ell^2$ and $C=\ell^2_+$ (for which $\inte C=\varnothing$).
\end{enumerate}
Then $(x_n)_\nnn$ is weakly convergent because
each condition guarantees that $\caff C =X$ 
and \cref{c:largeC} applies. 
We also note that \cref{c:largeC} is known
in the context of 
\fejer\ monotoncity 
(see \cite[Proposition~3.7(iii)]{Comb01F}). 
\end{remark}

The next result is somewhat surprising: the Opial set of an
Opial sequence can be enlarged to a closed 
affine subspace!

\begin{proposition}
\label{p:enlarge}
Let $(x_n)_\nnn$ be a sequence in $X$ that is 
Opial \wrt\ a nonempty subset $C$ of $X$. 
Then $(x_n)_\nnn$ is also Opial \wrt\ $\caff C$.
\end{proposition}
\begin{proof}
For $c\in C$, set 
$\ell(c) := \lim_n\|x_n-c\|$
which is well defined by assumption.

\emph{Step 1:}
 Extension to the affine hull.\\
We first show that the Opial set $C$ can be 
extended to $\aff C$. 
Pick a finite index set $I$ and take $c_i\in C$ for every 
$i\in I$. Also, let each $\lambda_i$ be real, with
$\sum_{i\in I}\lambda_i=1$. 
Using \cite[Lemma~2.14(ii)]{BC2017},
we see that when $n\to\infty$, we obtain 
\begin{align*}
\Big\|x_n-\sum_{i\in I}\lambda_i c_i\Big\|^2
&= \Big\|\sum_{i\in I}\lambda_i(x_n-c_i)\Big\|^2\\
&=\sum_{i\in I}\lambda_i\|x_n-c_i\|^2
-\thalb\sum_{i\in I,\,j\in I}\lambda_i\lambda_j\|(x_n-c_i)-(x_n-c_j)\|^2\\
&=\sum_{i\in I}\lambda_i\|x_n-c_i\|^2
-\thalb\sum_{i\in I,\,j\in I}\lambda_i\lambda_j\|c_i-c_j\|^2\\
&\to\sum_{i\in I}\lambda_i\ell^2(c_i)
-\thalb\sum_{i\in I,\,j\in I}\lambda_i\lambda_j\|c_i-c_j\|^2.
\end{align*}
as claimed\footnote{In passing, we point out 
that 
$\ell^2(\sum_i\lambda_ic_i)
= \sum_{i\in I}\lambda_i\ell^2(c_i)
-\thalb\sum_{i\in I,\,j\in I}\lambda_i\lambda_j\|c_i-c_j\|^2$. 
}. 

\emph{Step {\color{black} 2}:} Extension to the closure.\\
Our proof follows along lines of the proof of 
\cite[Proposition~3.2(i)]{BBCIRS}.
Let $\overline{c}\in \overline{C}$ say 
$c_m\to \overline{c}$, where
each $c_m$ lies in $C$. 
The triangle inequality yields
 the key inequality 
\begin{equation*}
(\forall m\in\NN)(\forall \nnn)\quad
-\|c_m-\overline{c}\|\leq \|x_n-\overline{c}\|-
\|x_n-c_m\|\leq \|c_m-\overline{c}\|,
\end{equation*}
from which the result will readily follow: indeed, 
keeping $m$ fixed and taking liminf/limsup as $n\to\infty$,
 we deduce that 
\begin{align*}
-\|c_m-\overline{c}\|
&\leq \varliminf_n \|x_n-\overline{c}\|-\ell(c_m)
\leq \varlimsup_n \|x_n-\overline{c}\|-\ell(c_m)
\leq \|c_m-\overline{c}\|. 
\end{align*}
Letting now $m\to\infty$ shows that 
$\varliminf_n \|x_n-\overline{c}\|
=\varlimsup_n \|x_n-\overline{c}\|
= \lim_m \ell(c_m)$. 
Hence $(\|x_n-\overline{c}\|)_\nnn$ converges and 
$\ell(\overline{c})=\lim_n\|x_n-\overline{c}\|=
\lim_m \ell(c_m)$ exists\footnote{We note the 
``continuity property''
$c_m\to \overline{c}$ $\Rightarrow$ 
$\ell(c_m)\to \ell(\overline{c})$.}.  

Now combine \emph{Step~1} and
\emph{Step~2} to obtain the conclusion. 
\end{proof}

\begin{remark} 
Some comments on \cref{p:enlarge} are in order. 
\begin{enumerate}
\item 
For a nonempty subset $C$ of $X$, we have the equivalences:
\begin{subequations}
\begin{align}
\text{$(x_n)_\nnn$ is Opial \wrt\ $C$} 
&\Leftrightarrow 
\text{$(x_n)_\nnn$ is Opial \wrt\ $\overline{C}$} \\
&\Leftrightarrow 
\text{$(x_n)_\nnn$ is Opial \wrt\ $\caff C$.} 
\end{align}
\end{subequations}
\item 
The counterpart for \fejer\ monotone sequence states
that if $(x_n)_\nnn$ is \fejer\ monotone \wrt\ $C$, 
then $(x_n)_\nnn$ is also \fejer\ monotone \wrt\ $\cconv C$. 
This is part of the folklore; see, e.g., 
\cite[Lemma~2.1(ii)]{BKLW}. 
\end{enumerate}
\end{remark}

\begin{example}
Suppose that $X=\RR^2$, and define 
$(x_n)_\nnn := ((-1)^n,0)_\nnn$.
Clearly, $(x_n)_\nnn$ is Opial \wrt\ 
$C := \{0\}\times[-1,1]$. 
By \cref{p:enlarge}, $(x_n)_\nnn$ is 
Opial \wrt\ $\aff C = \{0\}\times\RR$. 
Note that there is no proper superset $D$ of 
$\{0\}\times\RR$ such that $(x_n)_\nnn$ is Opial
\wrt\ $D$. This illustrates the sharpness of \cref{p:enlarge}. 
\end{example}

In the following two results, we describe precisely the relationship between 
sequences that are Opial \wrt\ $X$, and weakly and strongly convergent 
sequences.

\begin{corollary}
\label{c:X-Opial=>weak}
Let $(x_n)_\nnn$ be a sequence in $X$ that {\color{blue} is} Opial \wrt\ $X$. 
Then $(x_n)_\nnn$ is weakly convergent. 
\end{corollary}
\begin{proof}
Clear from \cref{c:Opialweakchar} with $C=X$. 
\end{proof}

\begin{proposition}
\label{p:weakvsX-Opial}
Let $(x_n)_\nnn$ be a weakly convergent sequence in $X$. 
Then $(x_n)_\nnn$ is Opial \wrt\ $X$
$\siff$ $(\|x_n\|)_\nnn$ is convergent.
\end{proposition}
\begin{proof}
Suppose $x_n\weakly \overline{x}$. Let's fix $x\in X$, write
\begin{equation*}
\|x_n-x\|^2 -\|x_n\|^2 = \|x\|^2 - 2\scal{x_n}{x}
\to \|x\|^2-2\scal{\overline{x}}{x},
\end{equation*}
which shows that 
$(x_n)_\nnn$ is Opial \wrt\ $X$
$\siff$ $\lim_n\|x_n\|$ exists 
($\siff$ $(x_n)_\nnn$ is Opial \wrt\ $\{0\}$).
\end{proof}

\begin{theorem}
\label{t:einaudi}
Let $(x_n)_\nnn$ be a sequence in $X$.
Then $(x_n)_\nnn$ is Opial \wrt\ $X$
$\siff$ $(x_n)_\nnn$ is weakly convergent 
and $(\|x_n\|)_\nnn$ is convergent. 
\end{theorem}
\begin{proof}
``$\RA$'': \cref{c:X-Opial=>weak} and \cref{p:weakvsX-Opial}.
``$\LA$'': \cref{p:weakvsX-Opial}.
\end{proof}




\begin{corollary}
\label{c:einaudi}
Suppose that $X$ is finite-dimensional and 
let $(x_n)_\nnn$ be a sequence in $X$.
Then $(x_n)_\nnn$ is Opial \wrt\ $X$
$\siff$ $(x_n)_\nnn$ is convergent. 
\end{corollary}
\begin{proof}
Recall that because $X$ is finite-dimensional, 
strong and weak convergence for sequences coincide. 
``$\RA$'': Clear from \cref{t:einaudi}. 
``$\LA$'': 
This is clear because every strongly convergent sequence is Opial 
\wrt\ $X$. 
\end{proof}

\begin{example}
\label{ex:einaudi}
Suppose that $X=\ell^2$ and 
let $(e_n)_\nnn$ be the sequence of standard unit vectors
in $X$. 
Then:
\begin{enumerate}
\item 
\label{ex:einaudi1}
$(e_n)_\nnn$ is Opial \wrt\ $X$, 
$e_n\weakly 0$; however, $e_n\not\to 0$ 
and $(e_n)_\nnn$ is \fejer\ only \wrt\ 
$\menge{c=(\gamma_n)\in X}{(\gamma_n)_\nnn\;\;\text{is increasing}}$. 
\item 
\label{ex:einaudi2}
$(x_n)_\nnn = (e_0,0,e_1,0,e_2,0,\ldots)$ 
converges weakly to $0$; however $(x_n)_\nnn$ is not 
Opial (and hence not \fejer\ monotone) 
\wrt\ any nonempty subset of $X$.  
\end{enumerate}
\end{example}
\begin{proof}
\cref{ex:einaudi1}: 
Clearly, $e_n\weakly 0$ and $\|e_n\|\equiv 1$.
By \cref{t:einaudi}, $(e_n)_\nnn$ is Opial 
\wrt\ $X$. 
Suppose that $(e_n)_\nnn$ is \fejer\ monotone
\wrt\ $\{c\}$, where $c=(\gamma_n)_\nnn \in X$. 
Then $(\forall\nnn)$
$\|e_{n+1}-c\|\leq\|e_n-c\|$
$\siff$
$\|e_{n+1}-c\|^2\leq\|e_n-c\|^2$
$\siff$
$1-2\scal{e_{n+1}}{c}+\|c\|^2 \leq 
1-2\scal{e_{n}}{c}+\|c\|^2$
$\siff$
$\scal{e_n-e_{n+1}}{c}\leq 0$
$\siff$
$\gamma_n-\gamma_{n+1}\leq 0$
$\siff$
$\gamma_n\leq\gamma_{n+1}$, as announced. 

\cref{ex:einaudi2}: 
Clearly, $x_n\weakly 0$. 
However, $(\|x_n\|)_\nnn = (1,0,1,0,1,0,\ldots)$ is not convergent. 
By \cref{t:einaudi}, $(x_n)_\nnn$ is not Opial \wrt\ $X$. 
Suppose to the contrary that $(x_n)_\nnn$ is Opial
\wrt\ $C$ and let $c\in C$.
Then $\ell := \lim_n\|x_n-c\|$ exists. 
Considering $(x_{2n+1})_\nnn$, we see that $\ell=\|c\|$. 
On the other hand, considering $(x_{2n})_\nnn$, we have
$\|c\|^2 \leftarrow \|x_{2n}-c\|^2 = 1+\|c\|^2 - 2\scal{x_{2n}}{c} \to 1+\|c\|^2$, which is absurd. 
\end{proof}

\begin{remark}
Let $(x_n)_\nnn$ be a sequence in $X$.
Then $(x_n)_\nnn$ is \fejer\ monotone
\wrt\ $X$ $\siff$ $(x_n)_\nnn = (x_0)_\nnn$ is a 
constant sequence. 
\cref{ex:einaudi}\cref{ex:einaudi1} illustrates 
a striking difference between Opial and \fejer\ monotone
sequences: if $(x_n)_\nnn$ is \fejer\ monotone \wrt\ a 
subset $C$ of $X$ with $\inte C\neq\varnothing$, 
then a result due to Ra\u{\i}k 
(see \cite{Raik} or \cite[Proposition~5.10]{BC2017}) implies 
that $(x_n)_\nnn$ converges 
strongly with the ``finite-length'' property
$\sum_n \|x_n-x_{n+1}\|<\pinf$. In contrast, the sequence 
$(e_n)_\nnn$ from \cref{ex:einaudi}\cref{ex:einaudi1}, 
which is Opial \wrt\ $X$, converges
only weakly! 
\end{remark}

\section{Opial sequences and projections onto Opial sets}

\label{s:three}

In this section, we study the behaviour of 
$(P_Cx_n)$ when $(x_n)_\nnn$ is a sequence in $X$
such that $(x_n)_\nnn$ is Opial \wrt\ the nonempty 
closed convex subset $C$ of $X$.
The closedness and convexity of $C$ is imposed because 
our analysis relies on the following 
concept, introduced by Edelstein \cite{Edel72} in 1972 
originally for analyzing nonexpansive mappings 
(see also \cite[Chapter~9]{GoeKirk}):

\begin{definition}[asymptotic center]
\label{d:ac}
Let $(u_n)_\nnn$ be a bounded sequence in $X$ 
and let $C$ be a nonempty closed convex subset of $X$.
Then there is a unique point $\widehat{c}\in C$ such that 
\begin{equation}
(\forall c\in C\smallsetminus\{\widehat{c}\})\quad
\varlimsup_n \|u_n-\widehat{c}\|<\varlimsup_n\|u_n-c\|.
\end{equation}
The point $\widehat{c}$ is the 
\emph{asymptotic center of $(u_n)_\nnn$ \wrt\ $C$}, 
and we 
write $\widehat{c} = A_C(u_n)_\nnn$ or, more succinctly, 
$A_C(u_{\NN})$. 
\end{definition}

We start with the following simple observation. 

\begin{lemma}
\label{l:simple1}
Let $(x_n)_\nnn$ be a sequence in $X$
that is Opial \wrt\ a 
nonempty closed convex subset $C$ of $X$.
Let $(x_{k_n})_\nnn$ be an arbitrary subsequence 
of $(x_n)_\nnn$.
Then these two sequences 
have the same asymptotic center \wrt\ $C$, i.e., 
$A_C(x_\NN) = A_C(x_{k_\NN})$.
\end{lemma}
\begin{proof}
Recall that $(x_n)_\nnn$ is bounded by \cref{p:opial}\cref{p:opial1}. 
Let $c\in C$.
Because $(x_n)_\nnn$ is Opial \wrt\ $C$, 
$(\|x_n-c\|)_\nnn$ is convergent and 
$\lim_n\|x_n-c\|=\lim_n\|x_{k_n}-c\|$.
Thus 
$\varlimsup_n\|x_n-c\|=\varlimsup_n\|x_{k_n}-c\|$
and so the asymptotic centers coincide.
\end{proof}

We also have the following properties:

\begin{fact}[asymptotic center] 
{\rm (See \cite[Proposition~11.18]{BC2017} and 
\cite[Theorem~3.3.2]{Cegielski}.)} 
\label{f:ac}
Let $(u_n)_\nnn$ be a bounded sequence in $X$ and
let $C$ be a nonempty closed convex subset of $X$.
Set $f\colon X \to \RR\colon x\mapsto \varlimsup_n\|x-u_n\|^2$. 
Then the following hold:
\begin{enumerate}
\item
\label{f:ac1}
The function $f+\iota_C$ is supercoercive, and strongly convex with constant $2$. 
\item 
\label{f:ac2}
$A_C(u_\NN)$, 
the \emph{asymptotic center} of $(u_n)_\nnn$ \wrt\ $C$, 
is the unique minimizer of $f+\iota_C$. 
\item 
\label{f:ac3}
If $u_n\weakly u$, then 
$(\forall x\in X)$ 
$f(x)=\|x-u\|^2 + f(u)$ and $P_Cu = A_C(u_\NN)$. 
\item 
\label{f:ac4}
If $(u_n)_\nnn$ is \fejer\ monotone \wrt\ $C$, 
then $P_Cu_n\to A_C(u_\NN)$. 
\end{enumerate}
\end{fact}

\begin{corollary}
\label{c:simple2}
Let $C$ be a nonempty closed convex subset of $X$, 
and let $(c_n)_\nnn$ be a sequence in $C$ 
that is Opial \wrt\ $C$.
Then $c_n\weakly A_C(c_\NN)$. 
\end{corollary}
\begin{proof}
By \cref{p:opial}\cref{p:opial1}, 
the sequence $(c_n)_\nnn$ is bounded. 
Let $\overline{c}\in C$ be a weak cluster point of 
$(c_n)_\nnn$, say $c_{k_n}\weakly \overline{c}$. 
On the one hand, \cref{f:ac}\cref{f:ac3} yields
$\overline{c}=P_C\overline{c} = A_C(c_{k_\NN})$. 
On the other hand, $A_C(c_\NN) = A_C(c_{k_\NN})$ 
by \cref{l:simple1}. 
Altogether, $\overline{c}=A_C(c_\NN)$ which yields the result.
\end{proof}

\begin{theorem}
\label{t:dto0}
Let $(x_n)_\nnn$ be a sequence in $X$ that is 
Opial \wrt\ a nonempty closed convex subset $C$ of $X$ and
such that $d_C(x_n)\to 0$. 
Then the asymptotic centers 
of $(x_n)_\nnn$ and $(P_Cx_n)_\nnn$ \wrt\ $C$ coincide, and 
$P_Cx_n\weakly A_C(x_\NN)$.  
\end{theorem}
\begin{proof}
By \cref{p:opial}\cref{p:opial1}, $(x_n)_\nnn$ is bounded.
Because $P_C$ is nonexpansive, it follows that 
$(P_Cx_n)_\nnn$ is bounded as well.
Because $x_n-P_Cx_n\to 0$, 
 we see that the inner product term in 
\begin{equation*}
(\forall c\in C)(\forall\nnn)\quad
\|x_n-c\|^2 = \|x_n-P_Cx_n\|^2 + \|P_Cx_n-c\|^2 
+ 2\scal{x_n-P_Cx_n}{P_Cx_n-c}
\end{equation*}
is controlled and goes to $0$. 
Now $(\|x_n-c\|^2)_\nnn$ converges, 
 and therefore 
$(\|P_Cx_n-c\|^2)_\nnn$ converges as well; 
moreover, 
\begin{equation}
(\forall c\in C)\quad \lim_n \|x_n-c\| = \lim_n\|P_Cx_n-c\|.
\end{equation}
This implies that $A_C(x_\NN) = A_C(P_Cx_\NN)$
and that $(P_Cx_n)_\nnn$ is Opial \wrt\ $C$. 
By \cref{c:simple2}, $P_Cx_n\weakly A_C(P_Cx_\NN)$
and we are done.
\end{proof}

\begin{remark}
Let us contrast \cref{t:dto0} with its \fejer\ counterpart:
if $(x_n)_\nnn$ is \fejer\ monotone \wrt\ a nonempty
closed convex subset $C$ of $X$, then
\cref{f:ac}\cref{f:ac4} yields the strong convergent statement 
$P_Cx_n\to A_C(x_\NN)$ even without the assumption 
that $d_C(x_n)\to 0$. 
On the other hand, \cref{ex:einaudi}\cref{ex:einaudi1} 
illustrates that the convergence in 
\cref{t:dto0} may fail to be strong: $P_Xe_n = e_n\weakly 0$ 
but $e_n\not\to 0$. 
Hence \cref{t:dto0} is sharp in the sense that 
it is impossible to obtain strong convergence of $(P_Cx_n)_\nnn$ even when $d_C(x_n)\to 0$. 
\end{remark}

\begin{proposition}
\label{p:badlung}
Let $(x_n)_\nnn$ be a sequence in $X$ that is 
Opial \wrt\ a nonempty closed convex subset $C$ of $X$.
Then all strong cluster points of $(P_Cx_n)_\nnn$
coincide with $A_C(x_\NN)$. 
\end{proposition}
\begin{proof}
Suppose that $P_{C}x_{k_n}\to\overline{c}$. 
Then 
$(\forall c\in C)$
$\|x_{k_n}-\overline{c}\|
\leq 
\|x_{k_n}-P_Cx_{k_n}\|
+\|P_Cx_{k_n}-\overline{c}\|
\leq 
\|x_{k_n}-c\|
+\|P_Cx_{k_n}-\overline{c}\|$. 
Because $(x_n)_\nnn$ is Opial \wrt\ $C$, we learn 
{\color{black} with the help of \cref{l:simple1}}  
that 
\begin{align*}
(\forall c\in C)\;\;
\lim_n\|x_n-\overline{c}\|
= 
\lim_n\|x_{k_n}-\overline{c}\|
\leq 
\lim_n\|x_{k_n}-c\|
= 
\lim_n\|x_{n}-c\|;
\end{align*}
therefore, 
$\overline{c}=A_C(x_\NN)$. 
\end{proof}

\begin{corollary}
\label{c:badlung}
Let $(x_n)_\nnn$ be a sequence in $X$ that is 
Opial \wrt\ a nonempty closed convex subset $C$ of $X$.
Suppose that $\{P_Cx_n\}_\nnn$ is relatively compact.
Then  $P_Cx_n\to A_C(x_\NN)$. 
\end{corollary}
\begin{proof}
Clear from \cref{p:badlung}. 
\end{proof}

\begin{corollary}
Let $(x_n)_\nnn$ be a sequence in $X$ that is 
Opial \wrt\ a nonempty closed convex subset $C$ of $X$.
Suppose that 
$\{x_n\}_\nnn$ is relatively compact, that 
$C$ is boundedly compact\footnote{Recall that 
$C$ is \emph{boundedly compact} if the intersection of $C$ with every closed ball 
is compact.}, or that 
$X$ is finite-dimensional.
Then  $P_Cx_n\to A_C(x_\NN)$. 
\end{corollary}
\begin{proof} 
Each assumption implies that 
$\{P_Cx_n\}_\nnn$ is relatively compact, and the 
conclusion thus follows from  \cref{c:badlung}. 
\end{proof}

Provided a relative compactness condition holds, 
we now have a result that sharpens 
\cref{p:opial}\cref{p:opial2}: 

\begin{proposition}
\label{p:saturday}
Let $(x_n)_\nnn$ be a sequence in $X$ that is 
Opial \wrt\ a nonempty closed convex subset $C$ of $X$.
Suppose that 
$\{P_Cx_n\}_\nnn$ is relatively compact. 
Then $(d_C(x_n))_\nnn$ is convergent and 
$\lim_n d_C(x_n) = \lim_n \|x_n-A_C(x_\NN)\|
= \min_{c\in C}\lim_n \|x_n-c\|$. 
\end{proposition}

\begin{proof}
Indeed, for all $n \in \NN$, we have: 
\begin{align*}
  d^2_C(x_n) &= \|x_n - P_C(x_n)\|^2\\
  &= \|x_n - A_C(x_{\NN})\|^2 + \|A_C(x_{\NN}) -  P_C(x_n)\|^2 + 2\langle x_n - A_C(x_n), A_C(x_{\NN}) -  P_C(x_n)\rangle.
\end{align*}
However, \cref{c:badlung} implies that the last two summands tend to zero as $\{P_Cx_n\}_\nnn$ is assumed to be relatively compact. 
The conclusion thus follows because 
$\lim_n \|x_n-A_C(x_\NN)\|$ in view of the Opial assumption \wrt\ $C$. 
\end{proof}

\begin{remark}
Consider \cref{p:saturday}. 
Without the relative compactness assumption, 
the conclusion that $\lim_n d_C(x_n) = 
\lim_n \|x_n-A_C(x_\NN)\|$ may fail: 
indeed, let's revisit \cref{ex:einaudi}\cref{ex:einaudi1}, 
where $(e_n)_\nnn$ is Opial \wrt\ $X$; hence, 
$d_X(e_n)\equiv 0$ while 
$A_X(e_\NN) = 0$ (by \cref{f:ac}\cref{f:ac3}) 
and so $\|e_n-A_C(e_\NN)\|\equiv 1$. 
In contrast, if $(x_n)_\nnn$ is \emph{\fejer\ monotone}
\wrt\ $C$, then $P_Cx_n\to A_C(x_\NN)$ by 
\cref{f:ac}\cref{f:ac4} and thus 
$\lim_n d_C(x_n) = \lim_n\|x_n-P_Cx_n\| = 
\lim_n \|x_n-A_C(x_\NN)\|$ without any
relative compactness assumption. 
\end{remark}

We now show that if the Opial set is affine, 
then the projected sequence is weakly convergent:

\begin{theorem}[projections onto an affine Opial set]
\label{t:nice}
Let $(x_n)_\nnn$ be a sequence in $X$ that is 
Opial \wrt\ a closed affine subspace $C$ of $X$.
Then $(P_Cx_n)_\nnn$ is weakly convergent 
to $A_C(x_\NN)$.
\end{theorem}
\begin{proof}
First, let $Y := C-C$ be the closed linear 
subspace parallel to $C$. 
Then $P_C\colon x\mapsto z + P_Yx$, where 
$z := P_C0 \in Y^\perp$. 
Next, let $c_1,c_2$ be two weak cluster points 
of $(P_Cx_n)_\nnn$, say 
\begin{equation*}
\text{
$P_Cx_{k_n}\weakly c_1$ 
and 
$P_Cx_{l_n}\weakly c_2$. 
}
\end{equation*}
After passing to subsequences and re-labelling if
necessary, we additionally assume that 
$x_{k_n}\weakly w_1$ and 
$x_{l_n}\weakly w_2$.
Because $P_Y$ is weakly continuous (\cite[Proposition~4.19(i)]{BC2017}), we deduce 
that $P_C$ is weakly continuous as well and thus 
$P_Cx_{k_n}\weakly c_1 = P_Cw_1 = z + P_Yw_1$ and 
$P_Cx_{l_n}\weakly c_2 = P_Cw_2 = z + P_Yw_2$. 
Hence 
\begin{equation}
\label{e:girlsback1}
c_1-c_2= P_Cw_1-P_Cw_2 = P_Y(w_1-w_2). 
\end{equation}
Furthermore,     
by \cref{t:old},
\begin{equation}
\label{e:girlsback2}
w_1-w_2 \in (C-C)^\perp = Y^\perp. 
\end{equation}
Combining \cref{e:girlsback1} with \cref{e:girlsback2}, 
we deduce that 
$c_1-c_2 = P_Y(w_1-w_2)\in P_Y(Y^\perp)=\{0\}$ 
and so $c_1=c_2$. 
We've shown that $(P_Cx_n)_\nnn$ is weakly 
convergent to $\overline{c} := c_1 = c_2$. 

Next, on the one hand, by \cref{l:simple1},
$A_C(x_\NN) = A_C(x_{k_\NN})$. 
On the other hand, by \cref{f:ac}\cref{f:ac3}, 
$\overline{c}=c_1=P_Cw_1 = A_C(x_{k_\NN})$. 
Altogether, 
$\overline{c} = A_C(x_\NN)$. 
\end{proof}

\begin{example}
\label{ex:nice}
Suppose that $X=\ell^2(\NN)$, with the standard basis 
of unit vectors $(e_n)_\nnn$. 
Define the sequence $(x_n)_\nnn$ in $X$ by 
\begin{equation}
(\forall\nnn)\quad x_n := 
\begin{cases}
e_0+e_1, &\text{if $n$ is even;}\\
e_1+e_{n+1}, &\text{if $n$ is odd}. 
\end{cases}
\end{equation}
Then the following hold:
\begin{enumerate}
\item 
\label{ex:nice1}
Set $Y := \{e_0\}^\perp$. 
The sequence 
$(x_n)_\nnn$ is Opial \wrt\ $Y$, 
but not \wrt\ $X$. 
The sequence $(P_Yx_n)_\nnn = (e_1,e_1+e_2,e_1,e_1+e_4,e_1,e_1+e_6,\ldots)$ converges weakly to 
$e_1 = A_Y(x_\NN)$ but 
$(P_Yx_n)_\nnn$  is not Opial \wrt\ $Y$. 
The sequence $(x_n-P_Yx_n)_\nnn
= (e_0,0,e_0,0,\ldots)$ does not converge weakly and 
the sequence $(d_Y(x_n))_\nnn = (1,0,1,0,\ldots,)$ is
not convergent either. 
\item 
\label{ex:nice2}
Set $B := \menge{x\in X}{\|x\|\leq 1}$ and 
$C := B\cap Y$. 
The sequence $(x_n)_\nnn$ is Opial \wrt\ $C$. 
But $(P_Cx_n)_\nnn = \tfrac{1}{\sqrt{2}}
(\sqrt{2}e_1,e_1+e_2,\sqrt{2}e_1,e_1+e_4,\sqrt{2}e_1,e_1+e_6,\ldots)$ is not weakly convergent and 
$d_C(x_n) = (1,\sqrt{2}-1,1,\sqrt{2}-1,\ldots)$ is not convergent either. 
\end{enumerate}
\end{example}
\begin{proof}
Observe first that $(x_n)_\nnn$ has exactly one 
strong cluster  point, namely $e_0+e_1$ and exactly
one different weak cluster point, namely $e_1$. 
In particular, $(x_n)_\nnn$ is not weakly convergent. 

\cref{ex:nice1}: 
Let $\nnn$ and $y\in Y$, i.e., $y\perp e_0$. 
If $n$ is even, we have 
$\|x_n-y\|^2
= \|(e_0+e_1)-y\|^2 = \|e_0+(e_1-y)\|^2
= \|e_0\|^2 + \|e_1-y\|^2 = 1+\|e_1-y\|^2$. 
And if $n$ is odd, we have 
$\|x_n-y\|^2
= \|(e_1+e_{n+1})-y\|^2
= \|(e_1-y)+e_{n+1}\|^2 
= \|e_1-y\|^2 + \|e_{n+1}\|^2 
+ 2\scal{e_1-y}{e_{n+1}}
= \|e_1-y\|^2 + 1 + 2\scal{e_1-y}{e_{n+1}}
\to \|e_1-y\|^2 +1$
because $e_{n+1}\weakly 0$. 
It follows that 
\begin{equation}
\|x_n-y\|^2 \to \|e_1-y\|^2+1,
\end{equation}
and so $(x_n)_\nnn$ is Opial \wrt\ $Y$. 
If $(x_n)_\nnn$ were Opial \wrt\ $X$, 
then $(x_n)_\nnn$ would be weakly convergent 
by \cref{c:X-Opial=>weak} which is absurd. 

If $n$ is even, we have 
$P_Yx_n = e_1$ 
and so $x_n-P_Yx_n= (e_0+e_1)-e_1 = e_0$; 
thus, $d_Y(x_n)=1$. 
If $n$ is odd, we have 
$x_n = e_1+e_{n+1}\in Y$ and 
so $P_Yx_n = x_n = e_1+e_{n+1} 
\weakly e_1$ and $d_Y(x_n) =0$.
Hence $P_Yx_n\weakly e_1$, 
which must be $A_Y(x_\NN)$ by \cref{t:nice}, 
and $(d_Y(x_n))_\nnn = (1,0,1,0,\ldots)$. 
The fact that $(P_Yx_n)_\nnn$ is not Opial \wrt\ $Y$
follows from \cref{p:Y}\cref{p:Y1} below. 

\cref{ex:nice2}: 
Because $C\subseteq Y$, it follows from \cref{ex:nice1} 
that $(x_n)_\nnn$ is Opial \wrt\ $C$. 
By \cite[Corollary~7.3]{BBW}, 
$P_C = P_B\circ P_Y$. 
Combing with \cref{ex:nice1} thus yields
\begin{equation}
P_Cx_n = 
\begin{cases}
e_1,&\text{if $n$ is even;} \\
\tfrac{1}{\sqrt{2}}(e_1+e_{n+1}), 
&\text{if $n$ is odd.}
\end{cases}
\end{equation}
If $n$ is even, then $x_n-P_Cx_n = (e_0+e_1)-e_1=e_0$
and so $d_C(x_n)=1$. 
If $n$ is odd, then 
$x_n-P_Cx_n = (1-1/\sqrt{2})(e_1+e_{n+1})$
which yields 
$d_C(x_n)= \sqrt{2}-1$. 
\end{proof}

\begin{remark}
Consider \cref{t:nice}. 
While $(P_Cx_n)_\nnn$ is weakly convergent, 
\cref{ex:nice}\cref{ex:nice1} illustrates that 
the sequence $(x_n-P_Cx_n)_\nnn$ may fail to converge weakly and the sequence $(d_C(x_n))_\nnn$ may
fail to converge as well. 
Furthermore, if we relax the assumption that $C$ be affine 
to mere convexity, 
\cref{ex:nice}\cref{ex:nice2} reveals
that $(P_Cx_n)_\nnn$ may fail to be weakly convergent. 
Thus the conclusions in \cref{t:nice} are sharp!
\end{remark}

\begin{proposition}
\label{p:Y}
Let $(x_n)_\nnn$ be a sequence in $X$, 
and let $Y$ be a closed linear subspace of $X$. 
Then the following hold:
\begin{enumerate}
\item 
\label{p:Y1}
Suppose that $(x_n)_\nnn$ is Opial \wrt\ $Y$.
Then: $(P_Yx_n)_\nnn$ is Opial \wrt\ $Y$
$\siff$ $(d_Y(x_n))_\nnn$ converges, in which case
$(P_Yx_n)$ is weakly convergent. 
\item 
\label{p:Y2}
Suppose that $(x_n)_\nnn$ is weakly convergent. 
Then:
$(P_Yx_n)_\nnn$ is Opial \wrt\ $Y$ 
$\siff$ $(\|P_Yx_n\|)_\nnn$ is convergent. 
\end{enumerate}
\end{proposition}
\begin{proof}
\cref{p:Y1}: 
We have $(\forall y\in Y)(\nnn)$ 
$\|x_n-y\|^2 = \|P_Yx_n-y\|^2 + d^2_Y(x_n)$.
Because $\lim_n \|x_n-y\|$ exists, 
the announced equivalences holds. 
In this case, 
\cref{c:X-Opial=>weak} gives the weak convergence of 
$(P_Yx_n)_\nnn$ (in $Y$), which in turn gives
the weak convergence of $(P_Yx_n)_\nnn$ (in $X$).

\cref{p:Y2}: 
Suppose that $x_n\weakly x \in X$, and write 
\begin{align*}
(\forall y\in Y)(\forall\nnn)\quad
\|P_Yx_n-y\|^2
&=
\|P_Yx_n\|^2-\|x\|^2 + \|x-y\|^2
+2\scal{x-P_Yx_n}{y}.
\end{align*}
Note that $\scal{x-P_Yx_n}{y} = \scal{x-x_n}{y}\to 0$.
Thus $(P_Yx_n)_\nnn$ is Opial \wrt\ $Y$ 
$\siff$
$(\|P_Yx_n\|)_\nnn$ is convergent. 
\end{proof}

\begin{example}
\label{ex:nolinear}
Suppose that $X=\RR^2$ and set $Y:=\RR\times\{0\}$.
Let $(\varepsilon_n)_\nnn$ be a strictly decreasing 
sequence 
of positive real numbers with $\varepsilon_n\to 0$. 
Define the sequence 
$(x_n)_\nnn := (\varepsilon_n,(-1)^n)_\nnn$ in $X$. 
Then $(x_n)_\nnn$ is Opial \wrt\ $Y$ because 
for every $y=(\eta,0)\in Y$, we have 
$\|x_n-y\|^2 = (\varepsilon_n-\eta)^2 + 1
\to \eta^2+1$. 
There is no proper superset $D$ of $Y$ such that 
$(x_n)_\nnn$ is Opial \wrt\ $D$. 
Moreover, $P_Y(x_n)=(\varepsilon_n,0)\to (0,0) 
\neq (\varepsilon_0,0) = P_Y(x_0)$ 
and $(x_n-P_Yx_n)_\nnn = (0,(-1)^n)_\nnn$ does not converge. 
\end{example}

\begin{remark}
If a sequence $(x_n)_\nnn$ is \fejer\ monotone 
\wrt\ a closed linear subspace $Y$ of $X$, 
then it is well known that $P_Yx_n\equiv P_Yx_0$ 
(see, e.g., \cite[Proposition~5.9(i)]{BC2017}). 
\cref{ex:nolinear} illustrates that this 
projection property
does not hold for Opial sequences. 
\end{remark}

\section*{Acknowledgments}
The authors thank the co-editor, Dr.~Andrea~Lodi, and 
two anonymous reviewers for their very careful reading, and 
helpful and constructive comments. 
We also thank 
Dr.\ Juan Peypouquet for his comments on 
Opial's Lemma and 
Dr.\ Patrick Combettes for providing us with a copy of 
\cite{BaillonThesis}. 
The research of HHB was supported by a Discovery Grant from 
 the Natural Sciences and Engineering Research Council of Canada. 
 
\section*{Declarations}

\textbf{Conflict of interest:}
There are no financial or non-financial interests that 
are directly or indirectly related to this work.


\begin{thebibliography}{999}
 \seppthree

\small


\bibitem{Fstarlong}
A.\ Arakcheev and H.H.\ Bauschke:
Fej\'er and Fej\'er* monotonicity:
new results and limiting examples,
preprint, December 2025.
\url{https://arxiv.org/abs/2512.17039}



\bibitem{AttCab}
H.\ Attouch and A.\ Cabot:
Convergence of a relaxed inertial proximal algorithm
for maximally monotone operators,
\emph{Mathematical Programming (Series A)}~184 (2020), 243--287. 
\url{https://doi.org/10.1007/s10107-019-01412-0}



\bibitem{BaillonThesis}
J.-B.\ Baillon:
\emph{Comportement asymptotique des contractions et semi-groupes de contraction}, 
Th\`ese. Universit\'e Paris~6, 1978.

\bibitem{HBThesis}
H.H.\ Bauschke:
\emph{Projection Algorithms and Monotone Operators},
PhD thesis, Simon Fraser University, 1996.
\url{https://summit.sfu.ca/item/7015}

\bibitem{BBW}
H.H.\ Bauschke, M.N.\ B\`{u}i, and X.\ Wang:
Projecting onto the intersection of a cone and a sphere, 
\emph{SIAM Journal on Optimization}~28 (2018), 2158--2188. 

\bibitem{BC2017}
H.H.\ Bauschke and P.L.\ Combettes: 
\emph{Convex Analysis and Monotone Operator Theory in Hilbert Spaces},
2nd edition, Springer, 2017.
\url{https://doi.org/10.1007/978-3-319-48311-5}

\bibitem{BKLW}
H.H.\ Bauschke, M.\ Krishan Lal, and X.\ Wang:
Directional asymptotics of Fej\'er monotone sequences,
\emph{Optimization Letters}~17 (2023), 531--544.
\url{https://doi.org/10.1007/s11590-022-01896-4}




\bibitem{BBCILS}
R.\ Behling, Y.\ Bello-Cruz, A.N.\ Iusem,
D.\ Liu, and L.-R.\ Santos:
A finitely convergent circumcenter method for 
the convex feasibility problem, 
\emph{SIAM Journal on Optimization}~34 (2024), 2535--2556.
\url{https://doi.org/10.1137/23M1595412}


\bibitem{BBCIRS}
R.\ Behling, Y.\ Bello-Cruz, A.N.\ Iusem,
A.A.\ Ribeiro, L.-R.\ Santos:
Fej\'er* monotonicity in optimization algorithms,
arXiv, 2024.  
\url{https://doi.org/10.48550/arXiv.2410.08331}


\bibitem{BCN25}
R.I.\ Bo\c{t}, E.R.\ Csetnek, and D.-K.\ Nguyen:
Fast Optimistic Gradient Descent Ascend (OGDA) Method 
in continuous and discrete time,
\emph{Foundations of Computational Mathematics}~25 (2025), 
163--222.
\url{https://doi.org/10.1007/s10208-023-09636-5}


\bibitem{Browder67}
F.E.\ Browder:
Convergence theorems for sequences of nonlinear 
operators in Banach spaces, 
\emph{Mathematische Zeitschrift}~100 (1967), 201--225. 


\bibitem{Cegielski}
A.\ Cegielski:
\emph{Iterative Methods for Fixed Point Problems in 
Hilbert Spaces}, Springer, 2012. 


\bibitem{Comb01F}
P.L.\ Combettes: Fej\'er-monotonicity in convex optimization.
In \emph{Encyclopedia of Optimization}, 
(C.A.\ Floudas and P.M.\ Pardalos, editors), 
pp.~106--114, Springer, 2001. 

\bibitem{Comb01qF}
P.L.\ Combettes: Quasi-Fej\'erian analysis of some optimization algorithms. In \emph{Inherently Parallel Algorithms in Feasibility and Optimization and Their Applications}, 
(D.\ Butnariu, Y.\ Censor, and S.\ Reich, editors), 
pp.~115--152, Elsevier, 2001.


\bibitem{Edel72}
M.\ Edelstein:
The construction of an asymptotic center
with a fixed-point property,
\emph{Bulletin of the American Mathematical Society}~78(2) (1972), 
206--208. 

\bibitem{GoeKirk}
K.\ Goebel and W.A.\ Kirk:
\emph{Topics in Metric Fixed Point Theory},
Cambridge University Press, 1990. 


\bibitem{NePo}
M.\ Neri and T.\ Powell:
A quantitative Robbins-Siegmund theorem,
\url{https://arxiv.org/abs/2410.15986}.

\bibitem{Opial}
Z.\ Opial:
Weak convergence of the sequence of successive 
approximations for nonexpansive mappings,
\emph{Bulletin of the AMS}~73 (1967), 591--597.



\bibitem{Peypouquet}
J.\ Peypouquet:
\emph{Convex Optimization in Normed Spaces}, 
Springer, 2015. 

\bibitem{Polyak}
B.T.\ Polyak:
\emph{Introduction to Optimization}, 
Optimization Software, 1987.
Available for download as djvu file 
\url{http://drive.google.com/uc?export=download&id=16eacYys8m4oFzdv9xxHVzo1TbTc-u6dl}
from \url{https://sites.google.com/site/lab7polyak/}.

\bibitem{Raik}
\`E.\ Ra\u{\i}k:
A class of iterative methods with {F}ej\'er-monotone
sequences,
\emph{Eesti NSV Teaduste Akadeemia Toimetised. 
F\"u\"usika-Matemaatika. 
Izvestiya Akademii Nauk \`Estonsko\u{\i}\ SSR. 
Seriya Fizika-Matematika}~18 (1969), 22--26.
In Russian. 


\bibitem{RoSi}
H.\ Robbins and D.\ Siegmund:
A convergence theorem for non negative almost supermartingales
and some applications. 
In \emph{Optimizing Methods in Statistics}, Elsevier, 1971, 
pp.~233--257.



\bibitem{Schafer}
H.\ Sch\"afer:
\"Uber die Methode sukzessiver Approximationen,
\emph{Jahresbericht der Deutschen Mathematiker-Vereinigung}~59 (1957), 131--140.



\end{thebibliography}
\end{document}